\documentclass{ourlematema}

\usepackage[capitalize]{cleveref}
\crefname{thm}{Theorem}{Theorems}
\usepackage{xcolor}
\usepackage{graphics}
\usepackage{graphicx}
\usepackage{import}
\usepackage{tikz}
\usepackage{enumitem}
\usetikzlibrary{arrows,automata}
\usepackage[outline]{contour}
\usepackage{blkarray, bigstrut}
\contourlength{1.2pt}

\newcounter{mainthmcnt}

\title{Reciprocal Maximum Likelihood Degrees of Brownian Motion Tree Models}

\date{14/05/2021}

\titlemark{Reciprocal ML-Degrees of Brownian Motion Tree Models}

\titlenote{
T.B.\ is funded by the Deutsche Forschungsgemeinschaft (DFG, German Research Foundation) -- 314838170, GRK 2297 MathCoRe.
J.I.C.\ is partially supported by the US National Science Foundation
(DGE 1746939).
C.E.\ is partially supported by the US National Science Foundation (DMS-2001854).
}

\titlenote{Acknowledgements:
We thank Piotr Zwiernik and Carlos Am\'endola for suggesting the problem, and Tim Seynnaeve, Caroline Uhler and Carlos Am\'endola for helpful discussions. We also thank the organizers of the Linear Spaces of Symmetric Matrices working group at MPI MiS Leipzig.}

\MSC{62R01, 14M25, 62F10}
\keywords{Brownian motion tree model, maximum likelihood degree, toric fiber product}

\author{T. Boege}
\address{%
MPI for Mathematics in the Sciences, Germany \\
\email{tobias.boege@mis.mpg.de}
}
\author{J. I. Coons}
\address{%
St John’s College, University of Oxford, United Kingdom \\
\email{coons@maths.ox.ac.uk}}

\author{C. Eur}
\address{%
Harvard University, USA \\
\email{ceur@math.harvard.edu}
} 
\authorline

\author{A. Maraj}
\address{%
University of Michigan, USA \\
\email{maraja@umich.edu}
}

\author{F. R\"ottger}
\address{Universit{\'e} de Gen{\`e}ve, Switzerland \\
\email{frank.roettger@unige.ch}
}
\authormark{T. Boege \ - \ J.I. Coons \ -\  C. Eur \ -\ A. Maraj \ -\ F. R\"ottger}
\date{2007/01/01}

\newcommand{\trace}{\text{trace}}

\newcommand{\bfp}{\mathbf{p}}

\newcommand{\bfu}{\mathbf{u}}

\newcommand{\ZZ}{\mathbb{Z}}
\newcommand{\CC}{\mathbb{C}}

\newcommand{\Acal}{\mathcal{A}}
\newcommand{\Lcal}{\mathcal{L}}
\newcommand{\Pcal}{\mathcal{P}}
\newcommand{\cL}{\mathcal{L}}
\newcommand{\cLI}{\mathcal{L}^{-1}}

\newcommand{\Mcal}{\mathcal{M}}
\newcommand{\Hcal}{\mathcal{H}}

\newcommand{\rowspan}{\mathrm{rowspan}}
\newcommand{\lca}{\mathrm{lca}}
\newcommand{\de}{{\rm des}}

\newcommand{\mld}{\mathrm{mld}}
\newcommand{\rmld}{\mathrm{rmld}}
\newcommand{\Lv}{\mathrm{Lv}}
\newcommand{\In}{\mathrm{Int}}
\newcommand{\outdeg}{\mathrm{outdeg}}
\newcommand{\ol}{\overline}
\begin{document}
\maketitle

\begin{abstract}
 We give an explicit formula for the reciprocal maximum likelihood degree of Brownian motion tree models.
To achieve this, we connect them to certain toric (or log-linear) models, and express the Brownian motion tree model of an arbitrary tree as a toric fiber product of star tree models.
\end{abstract}

\section{Introduction}
Let $T$ be a rooted tree on leaves $0, \dots, n$ with the leaf labeled $0$ as the root and with all edges directed away from the root. We denote the set of leaves of $T$ by $\Lv(T) = \{0, \dots, n\}$ and the set of internal vertices of $T$ by $\In(T)$. The \emph{out-degree} of vertex $v$, denoted $\outdeg(v)$, is the number of edges directed out of $v$.
For two leaves $i$ and $j$, denote their most recent common ancestor by $\operatorname{lca}(i,j)$.
We assume that $T$ does not  have any vertices of degree two.

\pagebreak 

The Brownian motion tree model on $T$  identifies the non-root leaves of the tree with random variables that are jointly distributed according to a multivariate Gaussian distribution with mean 0. To each vertex $v$, it assigns a parameter $t_v$ such that the covariance of two non-root leaves $i$ and $j$ is $t_{\text{lca}(i,j)}$.
In other words, this model is a linear Gaussian covariance model $\Mcal_T = \Lcal_T \cap \mathbb{S}_{>0}^n$, where $\mathbb{S}^n_{>0}$ is the set of $n \times n$ positive-definite matrices and $\Lcal_T$ is the  subspace of the space of symmetric $n\times n$ matrices $\mathbb{S}^n$ defined by
\[
\Lcal_T=\{ \Sigma \in \mathbb{S}^n \mid \sigma_{ij}=\sigma_{kl}~\text{if}~\text{lca}(i,j)=\text{lca}(k,l) \}.
\]
An example tree and its induced covariance pattern are shown in \Cref{fig:1}.
This model is a Wiener process along $T$, and was first introduced by Felsenstein \cite{F1973} to model trait evolution along phylogenetic trees. For background material on this model and other methods for comparative phylogenetics, see \cite{HP1991}.
See \cite{SUZ2019} for a detailed analysis of the geometry of this model.\\

\begin{figure}[h]
    \centering
\includegraphics[scale=0.3]{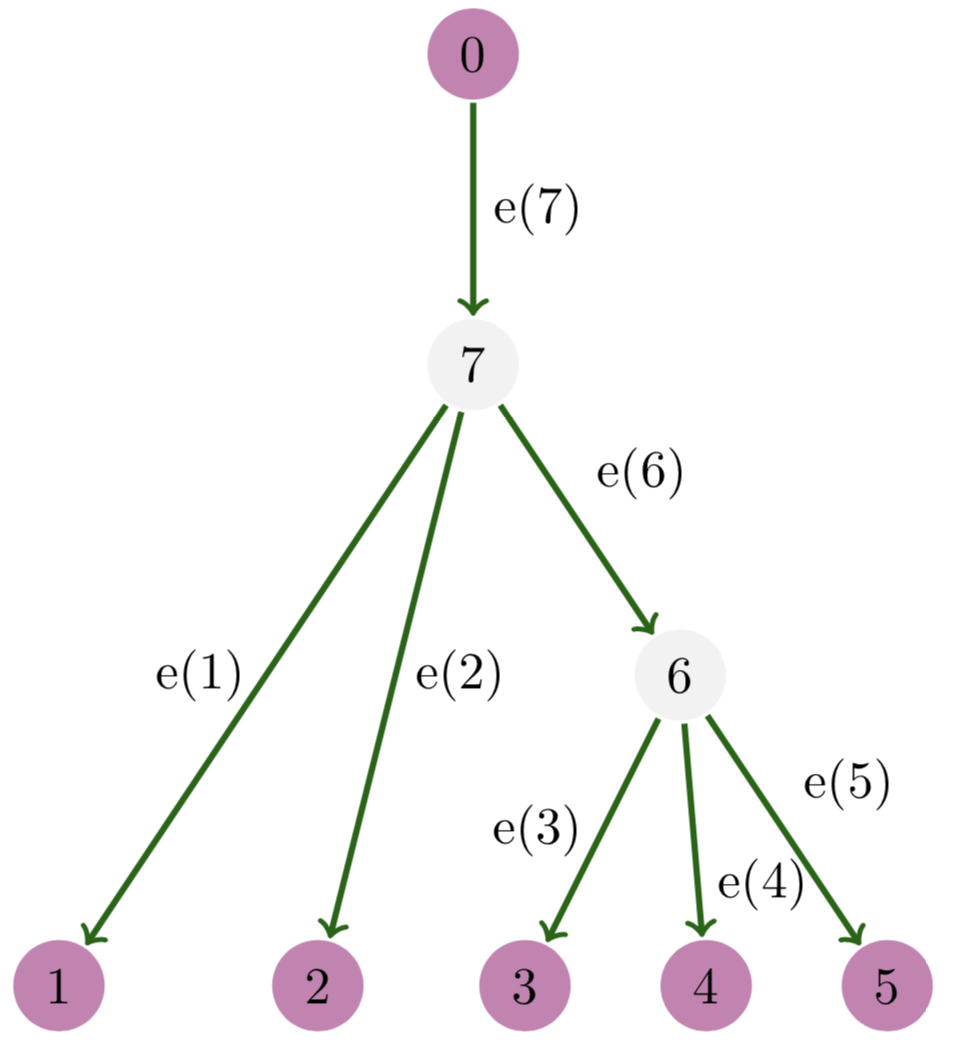}
\qquad
\includegraphics[scale=0.35]{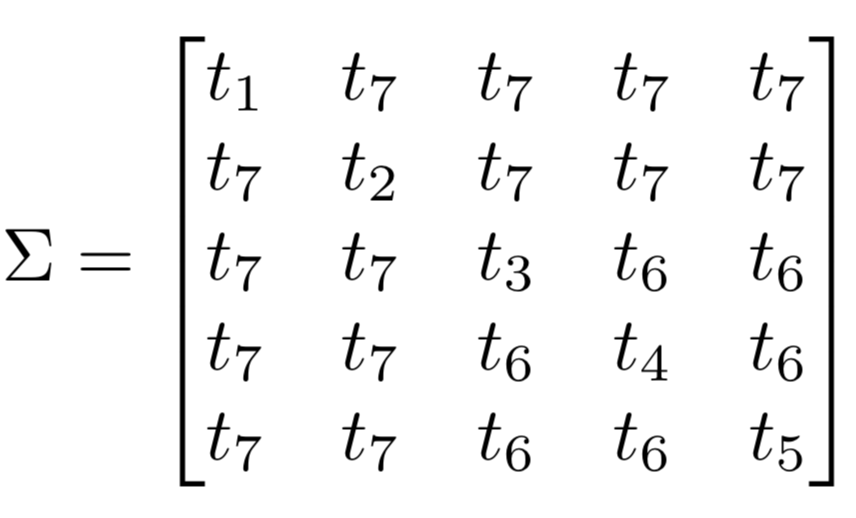}
    \caption{\textit{The given Brownian Motion Tree Model has reciprocal ML-degree $16$.}}
    \label{fig:1}
\end{figure}

In this paper we study properties of the reciprocal maximum likelihood estimation problem for Brownian motion tree models.
The \emph{log-likelihood function} of a linear Gaussian covariance model with an empirical covariance $S$ is the function
$\ell_S: \mathbb{S}_{>0}^n \to~\mathbb{R}$ defined by
\[
  \ell_S(\Sigma) = -\log \det(\Sigma) - \trace(S \Sigma^{-1}).
\]
The maximum likelihood estimator (MLE) is obtained by maximizing this log-likelihood function, which is equivalent to minimizing the Kullback-Leibler divergence $\text{KL}(S,\Sigma)$. To this optimization problem, one can associate a reciprocal problem which minimizes the ``wrong'' KL divergence $\text{KL}(\Sigma,S)$. This is  equivalent to maximizing the \emph{reciprocal log-likelihood function}:
\begin{align*}
    \ell_S^{\lor}(\Sigma) &= \log \det(\Sigma)-\trace (S^{-1} \Sigma).
\end{align*}
In the language of information theory, the standard MLE problem is obtained by performing the moment projection, or M-projection, of the data onto the statistical model, whereas the reciprocal MLE problem is obtained from the information projection, or I-projection \cite{neumann2011}.
We refer to~\cite[Section 3]{STZ19} and the references therein for more details. Our main interest is in the reciprocal maximum likelihood degree of these models.

\begin{dfn}[ML degree]
The \emph{maximum likelihood degree} of the model $\Mcal_T$, denoted $\mld(\Mcal_T)$, is the number of non-singular complex critical points of $\ell_S$ in parameters from the model $\Mcal_T$, counted with multiplicity, for generic symmetric $S$.
The \emph{reciprocal maximum likelihood degree}, denoted $\rmld(\Mcal_T)$, is defined analogously using the reciprocal likelihood $\ell^{\lor}_S$ in place of $\ell_S$.
\end{dfn}

\begin{rem}
There are different conventions in the literature for defining $\mld$ and $\rmld$ since a linear space of symmetric matrices can be viewed either as a space of covariance matrices or concentration matrices of a statistical model.
Our definitions of mld and rmld align with those in \cite{STZ19}, where the rmld is obtained by maximizing $\ell^{\lor}_S$ over the space of covariance matrices.
However, our notion of $\rmld$ coincides with that of $\mld$ in Section 4 of \cite{FMS2020}; this is because the authors of \cite{FMS2020} view $\Lcal$ as a space of concentration matrices.
\end{rem}

Knowledge of the ML-degree is useful for numerical methods in maximum likelihood estimation \cite{SW2005,STZ19}.
Our main result is a formula for the reciprocal ML-degree for Brownian motion tree models.

\begin{thm}\label{t:dualMLdegree}
The reciprocal ML-degree of the Brownian motion tree model $\Mcal_T$ is
\[
    \rmld(\Mcal_T) =\prod_{v\in \In(T)}(2^{\operatorname{outdeg}(v)}-\operatorname{outdeg}(v)-1).
\]
\end{thm}

For example, the reciprocal ML-degree of the tree model in \Cref{fig:1} is $16$, since the out-degrees of its two internal vertices are both $3$.

Our proof of \Cref{t:dualMLdegree} broadly consists of three steps.  
In~\cref{s:BMTM_and_loglinear}, we give preliminary definitions and theorems regarding toric models and the toric structure of the Brownian motion tree model as described in \cite{SUZ2019}. Then we show that the reciprocal maximum likelihood estimation problem in a Brownian motion tree model is equivalent to the standard maximum likelihood estimation problem of a toric model.
In~\cref{s:toricfiber} we show that this toric model has a toric fiber product structure as described in \cite{sullivant2007}, which implies that its ML-degree is the product of the ML-degrees of the models associated to two subtrees~\cite{AKK19}.
In~\cref{s:star_tree} we show that the reciprocal ML-degree of the Brownian motion tree model on a \emph{star tree} with $n+1$ leaves is $2^n - n -1$,
which serves as the base case for induction that completes the proof of~\cref{t:dualMLdegree}.

\section{Toric Models}\label{s:BMTM_and_loglinear}

A \emph{toric model}, also known as a \emph{log-linear model}, is a discrete statistical model whose Zariski closure is a toric variety \cite[Definition 6.2.1]{S2018}.
As such, it has a monomial parametrization, which is represented by an integral matrix $A \in \ZZ^{d\times m}$ called its \emph{design matrix}.  We assume throughout that $A$ has the vector of all ones in its rowspan. Its columns $\mathbf a_1, \ldots, \mathbf a_m$ define the monomial map
\begin{equation}\label{eqn:monomialmap}
\phi_A: \CC[p_1, \ldots, p_m] \to \CC[t_1^\pm, \ldots, t_d^\pm] \quad\text{which sends}\quad p_i \mapsto \mathbf{t}^{\mathbf{a}_i}.
\end{equation}
We denote by $I(A) \subset \CC[\bfp]$ the kernel of this map, and write $V(I(A)) \subseteq \CC^m$ for the toric affine subvariety defined by $I(A)$.

The maximum likelihood degree of a discrete statistical model is the number of complex critical points of the log-likelihood function counted with multiplicity \cite{ABBGHHNRS2019}.
In the case of toric models, it is the number of intersection points of the toric variety $V(I(A))$ with a specific affine linear space of complementary dimension.

\begin{prop}\textnormal{\cite[Proposition 7]{ABBGHHNRS2019}}\label{prop:toricMLD} Let $A \in \mathbb{Z}$ have the vector of all ones in its rowspan.
The maximum likelihood degree of a toric model $\mathcal M(A)$ with the design matrix $A$  is the number of solutions
\[
\bfp \in V(I(A)) \setminus V(p_1 \dots p_m (\textstyle \sum_{i=1}^mp_i)) \quad \text{satisfying} \quad A \bfp = A \bfu
\]
for generic data $\bfu \in \CC^m$, counted with multiplicity.
\end{prop}

In this section, we show that the reciprocal ML-degree of a Brownian motion tree model is equal to the ML-degree of a toric model.  Let $\Lcal_T^{-1}$ be the Zariski closure of $\{ \Sigma^{-1} \in \mathbb S^n \mid \text{$\Sigma \in \Lcal_T$ invertible} \}$. Our starting point is a result from \cite{SUZ2019} which states that $\Lcal_T^{-1}$ is toric under a linear change of coordinates.

Let $\Lcal_T^{-1} \subset \mathbb{S}^n$ with coordinates $K = (k_{ij})_{1 \leq i \leq j \leq n}$. Define new coordinates $\bfp = (p_{ij})_{0 \leq i < j \leq n}$ with change of coordinates $\bfp(K)$ given by
\begin{equation}\label{eq:changeOfCoordinates} 
\begin{split}
   p_{ij} &=-k_{ij} \quad\text{for}~1\le i< j\le n, \quad\text{and}\\
   p_{0i} &=\sum_{j=1}^n k_{ij} \quad\text{for}~1\le i \le n.
\end{split}
\end{equation}
The subscripts on each $p_{ij}$ are unordered sets; in other words, when $j>i$, we may write $p_{ji} = p_{ij}$.
Let $A_T \in \ZZ^{(|\operatorname{Vert}(T)|-1) \times \binom{n+1}{2}}$ be the matrix with rows corresponding to non-root vertices of $T$ and columns to pairs of leaves in $T$, defined by
\begin{equation}\label{eq:ATdef}
A_T(v,\{i, j\}) =
\begin{cases}
1 \text{ if } v = i \text{ or } v = j, \\
1 \text{ if } v = \lca(i,j), \\
0 \text{ otherwise.}
\end{cases}
\end{equation}
We can now state the key result from \cite{SUZ2019}.

\begin{thm}\label{thm:SUZmain}\textnormal{\cite[Theorem 1.2, Equation (10) \& Equation (11)]{SUZ2019}}
Let $\Lcal_T^{-1}$ be the Zariski closure of $\{ \Sigma^{-1} \in \mathbb S^n \mid \text{$\Sigma \in \Lcal_T$ invertible} \}$.  
After the linear change of coordinates $\bfp(K)$, the variety $\Lcal_T^{-1}$ is toric with defining matrix $A_T$. It is generated by the quadratic binomials,
\[
p_{ac}p_{bd} - p_{ad}p_{bc},
\]
where $a,b,c,d$ are distinct and $\{a,b\}$ and $\{c,d\}$ are the cherries of the 4-leaf subtree they induce.
\end{thm}

See Example \ref{eg:BT} for the matrix $A_T$ of the tree $T$ in \Cref{fig:1}.
We can now state the main result of this section.
\begin{thm}\label{Thm:DegreesEqual}
For a rooted tree $T$, the reciprocal ML-degree of the Brownian motion tree model on $T$ and the ML-degree of the toric model $\Mcal(A_T)$ are both equal to the degree of $V(I_T) \cap V(\langle A_T\bfp - A_T\bfu \rangle)$ for a generic choice of $\bfu$.
\end{thm}

The theorem can fail for linear covariance models not arising from Brownian tree models:
Example~\ref{eg:badToric} displays a linear subspace $\Lcal \subset \mathbb S^n$ of symmetric matrices such that $\Lcal^{-1}$, the Zariski closure of $\{\Sigma^{-1} \in \mathbb{S}^n \mid \Sigma \in \Lcal \text{ invertible} \}$, is a toric variety embedded in $\mathbb S^n$ via a monomial map, but the reciprocal ML-degree of the linear covariance model defined by $\Lcal$ is not equal to the ML-degree of the toric model defined by the embedded toric variety $\Lcal^{-1}$.

\medskip
We prepare the proof of Theorem~\ref{Thm:DegreesEqual} with two lemmas.  The first lemma is a standard computation in the maximum likelihood estimation of linear covariance models.  For a proof, see \cite[Proposition 3.3]{STZ19} or \cite[Equation (11)]{SU2010}. Endow the space of symmetric matrices $\mathbb{S}^n$ with the standard inner product $\langle A,B \rangle = \trace(A B)$.  For a linear subspace $\Lcal \subseteq \mathbb S^n$, denote by $\Lcal^\perp$ its orthogonal complement.

\begin{lemma}
\label{p:computeRmld}
The reciprocal ML-degree of the linear covariance model specified by $\cL$ is the number of solutions, counted with multiplicity, to the equations
\[
\Sigma \in \cL, \quad 
\Sigma K = \mathrm{Id}, \quad\text{and}\quad
K - S^{-1} \in \cL^{\perp}
\]
in the $2\cdot\binom{n+1}{2}$ entries of $\Sigma$ and $K$, for a generic choice of a sample concentration matrix $S^{-1}$.
\end{lemma}

The next lemma is a general geometric observation.

\begin{lemma}\label{lem:BadLocus}
Let $X$ be the vanishing locus in $\CC^n$ of a family of polynomials in $n$ variables, and suppose that $X$ has dimension $d$ with every $d$-dimensional irreducible component not contained in a hypersurface $H$.  Let $L\subset \CC^n$ be a linear subspace of dimension $n-d$.  Then, for a general $\overline{w}\in \CC^n/L$, the intersection $X \cap (L+w)$ lies in $X\setminus H$.
\end{lemma}

\begin{proof}
Since no $d$-dimensional component of $X$ is contained in $H$, we have $\dim (X\cap H) < d$. For each $w \in \CC^n$, let $\overline{w}$ denote the image of $w$ under the projection $\pi: \CC^n \rightarrow \CC^n / L$. The algebraic subset
$Z := \{\overline{w}\in \CC^n/L \mid (X \cap H) \cap (L+w) \neq \emptyset\}$ is the image of the restriction $\pi|_{X\cap H}$ of the projection map $\pi$ to $X \cap H$, since $\pi|_{X\cap H}$ maps $x\in X \cap H$ to the $\overline{w}\in \CC^n/L$ satisfying $x\in (L+w)$. Hence, we have $\dim Z \leq \dim (X\cap H) < d = \dim (\CC^n/L)$.  Thus, the set  $(\CC^n/L) \setminus Z$ is a nonempty Zariski dense subset of $\CC^n/L$, and any general $w\in \CC^n$ such that $\overline{w}\in (\CC^n/L) \setminus Z$ satisfies $X\cap (L+w) \subset X\setminus H$.
\end{proof}


\begin{exa}\label{eg:badToric}
Let $\Lcal$ be the set of all symmetric matrices of the form
\[
\begin{bmatrix}
a & c & c & c \\
c & b & 0 & 0\\
c & 0 & b & 0\\
c & 0 & 0 & b
\end{bmatrix}.
\]
Then the Zariski closure of the set of all inverses of elements of $\Lcal$ is
\[
\Lcal^{-1} = \Big\{ K\in \mathbb{S}^4 \mid k_{22} = k_{33} =k_{44}, k_{12} = k_{13}=k_{14}, k_{23}=k_{24}=k_{34}, k_{12}^2 = k_{11}k_{23}\Big\}.
\]
Thus $\Lcal^{-1}$ is toric.
One design matrix for the toric variety $\Lcal^{-1}$ is
\[
A = \begin{blockarray}{cccccccccc}
11 & 12 & 13 & 14 & 22 & 23 & 24 & 33 & 34 & 44 \\
\begin{block}{[cccccccccc]}
2 & 1 & 1 & 1 & 0 & 0 & 0 & 0 & 0 & 0\\
0 & 0 & 0 & 0 & 1 & 0 & 0 & 1 & 0 & 1\\
0 & 1 & 1 & 1 & 0 & 2 & 2 & 0 & 2 & 0\\
\end{block}
\end{blockarray}.
\]
Using Lemma \ref{p:computeRmld} and Proposition \ref{prop:toricMLD}, one can compute that the reciprocal ML-degree of the linear covariance model defined by $\Lcal$ is 1, whereas the ML-degree of the toric model $\Mcal(A)$ is 2. 
\end{exa}

The failure of Theorem~\ref{Thm:DegreesEqual} in the above example arises from the fact that the affine linear equations defining $K - S^{-1} \in \Lcal^{\perp}$ are not equivalent to those defining $A \bfp = A \bfu$. In the case of Brownian motion tree models, these affine linear equations are equivalent; showing this comprises much of the following proof of Theorem~\ref{Thm:DegreesEqual}.

\begin{proof}[Proof of \Cref{Thm:DegreesEqual}]
Lemma \ref{p:computeRmld} states that the reciprocal ML-degree of $\Mcal_T$ is the number of invertible matrices $K$ such that $K \in \Lcal_T^{-1}$ and $K - W \in \cL_T^{\perp}$ for a fixed generic $W\in \mathbb{S}^n$.
By \Cref{thm:SUZmain}, the first condition $K\in \Lcal_T^{-1}$ is equivalent to $\bfp(K) \in V(I_T)$.  The second condition $K - W \in \Lcal_{T}^{\perp}$ is equivalent to
\[
\sum_{\substack{1\leq i \leq j \leq n \\ \lca(i,j) = v}} (k_{ij} - w_{ij}) = 0 \quad\text{for each $v \in \operatorname{Vert}(T) \setminus \{0\}$}.
\]
Let $\bfu = \bfp(W)$.
This linear system is equivalent to
\begin{equation}\label{Eqn:BrownianLinearSystem}
\begin{split}
    \sum_{\substack{1\leq i \leq j \leq n \\ \lca(i,j) = v}} (p_{ij} - u_{ij}) = 0 & \text{ for each interior vertex } v \in \In(T), \text{ and}\\
    \sum_{\substack{j=0 \\ j \neq i}}^n (p_{ij} - u_{ij}) = 0 & \text{ for each leaf } i \in \Lv(T)\setminus \{0\}.
\end{split}
\end{equation}
This can be written as $A_T \bfp - A_T \bfu = \mathbf{0}$ with $A_T$ as defined in \Cref{eq:ATdef}.
Therefore the reciprocal ML-degree of the Brownian motion tree model on $T$ is the degree of the subscheme
\[
\big(V(I_T) \cap V(\langle A_T \bfp - A_T \bfu \rangle) \big) \setminus V(\det K) \subset \CC^{\binom{n+1}{2}}
\]
for a generic $\bfu$ where $\det K$ is written as a polynomial in the $\bfp$ coordinates.
Similarly, writing $\Hcal$ for the union of hyperplanes $V((\sum_{i,j} p_{ij}) \prod_{i,j} p_{ij})$, we have from Proposition \ref{prop:toricMLD} that the ML-degree of the toric model $\mathcal M(A_T)$ is the degree of the subscheme
\[
\big(V(I_T) \cap V(\langle A_T \bfp - A_T \bfu \rangle) \big) \setminus \Hcal \subset \CC^{\binom{n+1}{2}}.
\]
Note that $V(I_T)$ is contained in neither $V(\det K)$ nor $\Hcal$.
Indeed, the matrix of all ones is in $V(I_T) \setminus \Hcal$ and the identity matrix is in $V(I_T) \setminus V(\det K)$.
Lemma~\ref{lem:BadLocus} thus implies that for a generic $\bfu$,  the hypersurfaces $V(\det K)$ and $\Hcal$ do not intersect $V(I_T) \cap V(\langle A_T \bfp - A_T \bfu \rangle)$.
Therefore the reciprocal ML-degree of the Brownian motion tree model of $T$ and the ML-degree of $\Mcal(A_T)$ are both equal to the degree of $V(I_T) \cap V(\langle A_T \bfp - A_T \bfu \rangle)$.
\end{proof}

\section{Toric Fiber Products}\label{s:toricfiber}

To compute the ML-degree of the toric model $\Mcal(A_T)$, we show in this section that $I_T$ can be written as a toric fiber product of the ideals of two smaller trees, and consequently deduce that the ML-degree of $\Mcal(A_T)$ is a product of the ML-degrees of the toric models on these subtrees.
For background on the toric fiber product construction, see \cite{sullivant2007}.

We start by introducing a new parametrization of $I_T$ that makes the toric fiber product structure more apparent.
This parametrization is given by the matrix $B_T$ defined as follows.
Since every vertex of $T$ except for the root has in-degree 1, we label each edge of $T$ by $e(v)$ where $v$ is the vertex of $T$ that $e(v)$ is directed into.
Let $E(T)$ denote the edge set of $T$, and let $\Pcal(i,j) \subset E(T)$ denote the set of edges in the unique shortest path in $T$ between two leaves $i$ and $j$.
Define the matrix 
$B_T \in \ZZ^{E(T) \times \binom{n+1}{2}}$ by
\[
B_T(e, \{i,j\}) = \begin{cases}
1 & \text{ if $e \in \Pcal(i,j)$,} \\
0 & \text{ otherwise.}
\end{cases}
\]

\begin{prop}\label{prop:RowspansEqual}
For a rooted tree $T$, one has $\rowspan(A_T) = \rowspan(B_T)$.  In particular, the ideals $I(A_T)$ and $I(B_T)$ are equal.
\end{prop}

\begin{proof}
We show that matrix $B_T$ can be obtained by applying elementary row operations to $A_T$. Let  $a^v_T$ denote the row of $A_T$ corresponding to vertex $v$, and let $b_T^{e(v)}$ be the row in $B_T$ for edge $e(v)$. For vertex $v$, let $\de\Lv(v)$  be the set of all leaves descended from $v$, and let $\de\In(v)$  be the set of internal vertices descended from $v$.  The following holds. 
\begin{align}
    \label{eq:linearTransf}
   b^{e(v)}_T=\sum\limits_{\substack{k\in \de\Lv(v)}}a_T^{k}-2\sum\limits_{\substack{k\in \de\In(v)}}a_T^{k}.
\end{align}
Note that when $v$ is a leaf, $b_T^{e(v)} = a_T^v$.
The reader may wish to consult Example \ref{eg:BT} at this time.

Indeed, the edge $e(v)$ is in the unique shortest path between leaves $i$ and $j$ if and only if exactly one of these leaves is a descendent of $v$. Without loss of generality, let $i$ be this leaf. Then $i$ is in fact the only vertex descended from $v$ with nonzero $ij$-coordinate in row vectors $a^{i}_{T}$ appearing in \Cref{eq:linearTransf}. So the $ij$-coordinate of the right-hand side of \Cref{eq:linearTransf} is equal to 1. Now, suppose that  $e(v)$ is not in the unique shortest path between leaves $i$ and $j$. There are two cases to consider; either both $i$ and $j$ are descended from $v$, or neither of them are. In the former case, the vertices $k$ descended from $v$  with non-zero entries in the $ij$-coordinate of $a^k_T$ are $i,j$ and $\lca(i,j)$. Hence, the $ij$-coordinate of the right-hand side of \Cref{eq:linearTransf} is $0$. In the latter case, if both $i$ and $j$ are not descended from $v$, their least common ancestor is not in $\de\In(v)$. Hence, the right-hand side of \Cref{eq:linearTransf} is $0$.


Lastly, the two matrices have the same rank.  Indeed, the rank of $A_T$ is  $\mathrm{dim}(\cL_T^{-1})=\mathrm{dim}(\cL_T)=\#\mathrm{Vert}(T)-1.$  Take the set of columns $\{0,i\}$  in $B_T$ together with a column  $\{i_k,j_k\}$ with $\lca(i_k,j_k)=k$ for each internal node $k$ in $T$. This is a linearly independent set of $\#\mathrm{Vert}(T)-1$ vectors, which concludes that $\mathrm{rank}(B_T)\geq \mathrm{rank}(A_T)$. Combined with the fact that $\mathrm{rowspan}(B_T)\subseteq \mathrm{rowspan}(A_T)$, this implies that $A_T$ and $B_T$ have the same rowspan.
\end{proof}

\begin{exa}\label{eg:BT}
The matrix $A_T$ for the tree in Figure \ref{fig:1} is 
\[
{\footnotesize
\begin{blockarray}{cccccccccccccccc}
     & 01 & 02 & 03 & 04 & 05 & 12 & 13 & 14 & 15 & 23 & 24 & 25 & 34 & 35 & 45 \\
 \begin{block}{c[ccccccccccccccc]}
1 & 1 & 0 & 0 & 0 & 0 & 1 & 1 & 1 & 1 & 0 & 0 & 0 & 0 & 0 & 0\\
2 & 0 & 1 & 0 & 0 & 0 & 1 & 0 & 0 & 0 & 1 & 1 & 1 & 0 & 0 & 0\\
3 & 0 & 0 & 1 & 0 & 0 & 0 & 1 & 0 & 0 & 1 & 0 & 0 & 1 & 1 & 0\\
4 & 0 & 0 & 0 & 1 & 0 & 0 & 0 & 1 & 0 & 0 & 1 & 0 & 1 & 0 & 1\\
5 & 0 & 0 & 0 & 0 & 1 & 0 & 0 & 0 & 1 & 0 & 0 & 1 & 0 & 1 & 1\\
6 & 0 & 0 & 0 & 0 & 0 & 0 & 0 & 0 & 0 & 0 & 0 & 0 & 1 & 1 & 1\\
7 & 0 & 0 & 0 & 0 & 0 & 1 & 1 & 1 & 1 & 0 & 0 & 0 & 0 & 0 & 0\\
      \end{block}
\end{blockarray}.}
\]
The matrix $B_T$ for the tree in Figure \ref{fig:1} is 
\[
\footnotesize
\begin{blockarray}{cccccccccccccccc}
      & 01 & 02 & 03 & 04 & 05 & 12 & 13 & 14 & 15 & 23 & 24 & 25 & 34 & 35 & 45 \\
 \begin{block}{c[ccccccccccccccc]}
e(1) & 1 & 0 & 0 & 0 & 0 & 1 & 1 & 1 & 1 & 0 & 0 & 0 & 0 & 0 & 0\\
e(2) & 0 & 1 & 0 & 0 & 0 & 1 & 0 & 0 & 0 & 1 & 1 & 1 & 0 & 0 & 0\\
e(3) & 0 & 0 & 1 & 0 & 0 & 0 & 1 & 0 & 0 & 1 & 0 & 0 & 1 & 1 & 0\\
e(4) & 0 & 0 & 0 & 1 & 0 & 0 & 0 & 1 & 0 & 0 & 1 & 0 & 1 & 0 & 1\\
e(5) & 0 & 0 & 0 & 0 & 1 & 0 & 0 & 0 & 1 & 0 & 0 & 1 & 0 & 1 & 1\\
e(6) & 0 & 0 & 1 & 1 & 1 & 0 & 1 & 1 & 1 & 1 & 1 & 1 & 0 & 0 & 0\\
e(7) & 1 & 1 & 1 & 1 & 1 & 0 & 0 & 0 & 0 & 0 & 0 & 0 & 0 & 0 & 0\\
      \end{block}
\end{blockarray}.
\]
The following are the linear combinations of \Cref{eq:linearTransf}.
\begin{align*}
    &b^{e(i)}_T=a^i_T \text{ for } i=1,2,3,4, 5,\\
    &b^{e(6)}_T
    =a^3_T+a^4_T+a^5_T-2a^6_T
    =b^3_T+b^4_T+b^5_T-2a^6_T,  \text{ and }\\
    &b^{e(7)}_T
    =a^1_T+a^2_T+a^4_T +a^5_T-2a^6_T-2a^7_T
    =b^1_T+b^2_T+b^6_T-2a^7_T.
\end{align*}
\end{exa}

\medskip
In our computation of toric fiber products, it will be necessary to consider the ideal $I(B_T) \subset \CC[p_{ij} \mid 0\leq i < j \leq n]$ in a ring with one extra variable.  More precisely, let $B_T^{\star}$ be the matrix with rows indexed by $E \cup \{ \star \}$ and columns indexed by pairs of elements of $\{0, \dots, n\}$ and the symbol $\star$, whose entries are given by $B_T^{\star}(e, \{i,j\}) = B_T(e, \{i,j\})$, $B_T^{\star}(e,\star) = 0$ for all $e \in E$, $B_T^{\star}(\star, \{i,j\}) = 1$ for each $\{i,j\} \subset \{0,\dots,n\}$ and $B_T^{\star}(\star, \star) = 1$. In other words, $B_T^{\star}$ is obtained from $B_T$ by adding a column of all zeros and then a row of all ones.

\begin{rem}
\label{rem:MLdegreTFP}
Since the all-ones row vector $\mathbf{1}$ is in $\rowspan(B_T)$, the all-ones row $\ol{b}_T^{\star}$ in $B_T^\star$ can be replaced by the row consisting of all zeros except for the 1 in the ${\star}$ column without changing the ideal $I(B_T^\star)$.  Thus, the ideal $I(B_T^{\star})$ is the extension of the ideal $I(B_T) \subset \CC[p_{ij}\mid i,j \in \Lv(T)]$ in the ring with one extra variable $\CC[p_\star, p_{ij} \mid i,j \in \Lv(T)]$.
Consequently, the ML-degree of $I(B_T^{\star})$ is equal to that of $I(B_T)$. 
\end{rem}

Let us now consider a rooted tree $T$ built from two smaller trees in the following way.  Let $S_m$ be the rooted \emph{star tree}; that is, $S_m$ is a tree with a unique internal vertex on $m+1$ leaves. Let $T'$ be an arbitrary rooted tree.  Let $T$ be obtained from $T'$ and $S_m$ by identifying a distinguished leaf edge of $T'$ with the root edge of $S_m$. More precisely, let $\ell$ be a distinguished leaf of $T'$ with direct ancestor $h$. Label the root leaf of $S_m$ by $h$ and let $\ell$ label the unique internal vertex of $S_m$. We obtain $T$ from $T'$ and $S_m$ by identifying the vertices labeled $h$ and $\ell$ and the edge between them. \Cref{fig:2} illustrates such a procedure.  By identifying vertices $6$ and $7$ in the two trees, one obtains the tree in Figure \ref{fig:1}.\\

\begin{figure}[h]
\centering
\includegraphics[scale=0.38]{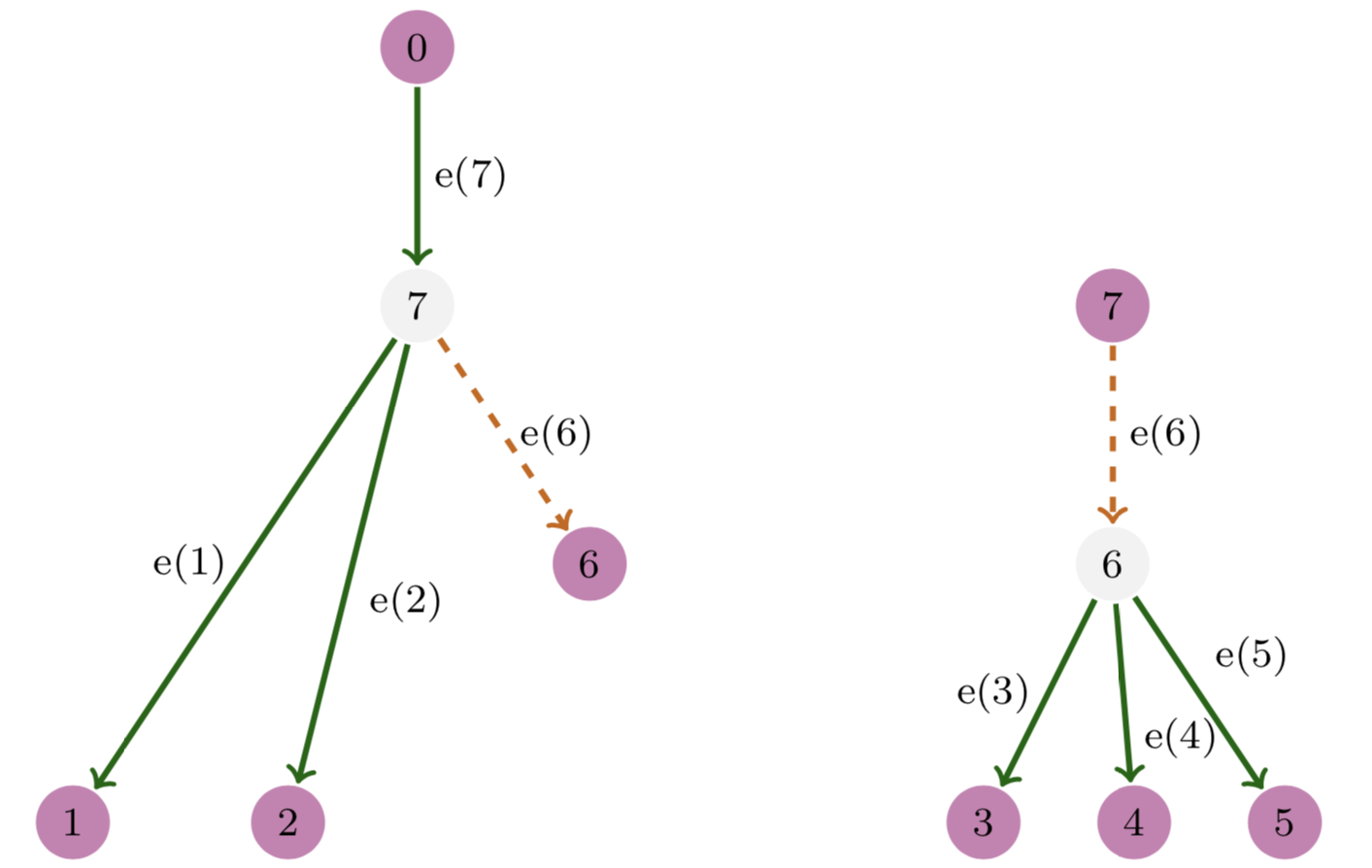}
\caption{\textit{Identifying vertices $6$ and $7$ in these trees produces the tree in Figure~\ref{fig:1}}}
\label{fig:2}
\end{figure}

Let $\CC[\bfp] = \CC[p_{i,j} \mid i, j \in (\Lv(T') \cup \Lv(S_m))\setminus \{h, \ell\}, i \neq j]$, $\CC[\mathbf{q}] = \CC[q_{i,j} \mid i,j\in \Lv(T'), i \neq j]$ and $\CC[\mathbf{r}] = \CC[r_{i,j}\mid i,j\in \Lv(S_m), i \neq j]$ We will show that the ideal, $I(B_T)\subset \CC[\bfp]$
is a toric fiber product of the two ideals $I(B_{T'}^{\star}) \subset \CC[q_{\star}, \mathbf{q}]$ and $I(B_{S_m}^{\star}) \subset \CC[r_{\star}, \mathbf{r}]$.
Following the definition of the toric fiber product in \cite{sullivant2007}, we assign a multigrading to the indeterminates of the polynomial rings associated to $T'$ and $S_m$ as follows.
Assign the following multidegrees to the variables of $\mathbb{C}[q_{\star}, \mathbf{q}]$
\begin{align*}
\deg(q_{\star}) &= [0,0,1], &  \deg(q_{i,j}) &= \begin{cases}
[1,0,0] & \text{ if } i, j \neq \ell, \\
[0,1,0] & \text{ if } i = \ell \text{ or } j = \ell.
\end{cases}
\end{align*}
Similarly, let
\begin{align*}
\deg(r_{\star}) &= [1,0,0], & \deg(r_{i,j}) &= \begin{cases}
[0,0,1] & \text{ if } i, j \neq h, \\
[0,1,0] & \text{ if } i = \ell \text{ or } j = h.
\end{cases} 
\end{align*}
Finally, let
\[
\deg(p_{i,j}) = \begin{cases}
[1,0,0] & \text{ if } i,j \in L(T'), \\
[0,0,1] & \text{ if } i,j \in L(S_m), \\
[0,1,0] & \text{ otherwise.}
\end{cases}
\]
Then the matrix $\Acal$ whose rows are these multigrading vectors is the $3 \times 3$ identity matrix and hence has full rank. 

\begin{prop}\label{prop:multihom}
The ideals $I(B_{T'}^*)$ and $I(B_{S_m}^*)$ are multi-homogeneous with respect to the given multigradings.
\end{prop}

\begin{proof}
The generators of $I(B^*_{T'})$ are identical to those of $I(B_{T'})$ and each generator of $I(B_{T'})$ has the form $p_{ac}p_{bd} - p_{ad}p_{bc}$, as described in \cref{thm:SUZmain}. At most one of $a,b,c,d$ may be equal to $\ell$. If none are equal to $\ell$, then the multidegree of each monomial is $[2,0,0]$. If exactly one is equal to $\ell$, then the multidegree of each monomial is $[1,1,0]$. Since each generator of $I(B^*_{T'})$ is multi-homogeneous with respect to the given multigrading, the ideal itself is also multi-homogeneous. The argument for $I(B^*_{S_m})$ is analogous.
\end{proof}

Proposition~\ref{prop:multihom} allows us to define the toric fiber product of the ideals $I(B_{T'}^*)$ and $I(B_{S_m}^*)$. Let $R_{T'} = \CC[q_{\star}, \mathbf{q}] / I(B_{T'}^{\star})$ and let $R_{S_m} = \CC[r_{\star}, \mathbf{r}] / I(B_{S_m}^{\star})$. 
With respect to these multigradings, the {toric fiber product} of $I(B_{T'}^{\star})$ and $I(B_{S_m}^{\star})$, denoted as $I(B_{T'}^{\star})\times_{\Acal}I(B_{S_m}^{\star})$ is the kernel of the map,
\begin{alignat*}{1}
\psi_{T',S_m} \colon & \CC[\bfp] \rightarrow R_{T'} \otimes_{\CC} R_{S_m} \\
&\begin{cases}
p_{i,j}  \mapsto  q_{i,j} \otimes r_{\star} & \text{ if } i,j \in Lv(T')\setminus \{\ell\}, \\
p_{i,j}  \mapsto  q_{\star} \otimes r_{i,j} & \text{ if } i,j \in Lv(S_m)\setminus \{ h \}, \text{ and} \\
p_{i,j} \mapsto q_{i,\ell} \otimes r_{h,j}, & \text{ if } i \in Lv(T') \setminus \{\ell\} \text{ and } j \in Lv(S_m) \setminus \{ h \}.
\end{cases}
\end{alignat*}

\begin{rem}
Combinatorially, this operation corresponds to including paths between leaves of the smaller trees $T'$ and $S_m$ into $T$. Paths whose leaves are both in $T'$ or $S_m$ remain the same, whereas we glue together paths in $T'$ and $S_m$ with endpoints $\ell$ and $h$ respectively along their common edge.
\end{rem}

\begin{thm}\label{thm:toricfibertree}
With the notation as above, we have $I(B_T) = I(B_{T'}^{\star}) \times_{\Acal} I(B_{S_m}^{\star}).$
\end{thm}

\begin{proof}

We may rewrite the map defining the toric fiber product as
\begin{alignat*}{3}
\psi_{T',S_m} \colon \quad &\CC[\bfp] & \rightarrow & \CC[t_{\star}, t_e \mid e \in E(T)] \\
&p_{i,j} &\mapsto & t_{\star} \big( \prod_{e \in \Pcal(i,j) \cap E(T')} t_e \big) t_{\star} \big( \prod_{e \in \Pcal(i,j) \cap E(S_m)} t_e \big).
\end{alignat*}
Note that $t_{\star}$ and $t_{e(\ell)}$ are always squared in the image of this map. Indeed, $t_{\star}^2$ is a factor of each $p_{ij}$. The parameter $t_{e(\ell)}$ does not appear as a factor of $p_{ij}$ when the path $\Pcal(i,j)$ lies entirely within $T'$ or $S_m$. When $i$ is a leaf of $T'$ and $j$ is a leaf of $S_m$ (or vice versa), $t_{e(\ell)}^2$ divides $p_{ij}$. So we may replace the parameters $t_{\star}$ and $t_{e(\ell)}$ with their square roots without changing the kernel of $\psi_{T',S_m}$. After this replacement, the row corresponding to $t_*$ in the matrix defining $\psi_{T',S_m}$ is the row of all ones. Since the row of all ones is in $\rowspan(B_T)$, the kernel of $\psi_{T', S_m}$ is equal to the kernel of the map $\phi_{B_T}$ associated to $B_T$ as in \Cref{eqn:monomialmap}.
\end{proof}

\begin{cor}
\label{cor:MLdegreToricFiberProduct}
The ML-degree of   $I(B_T)$ is equal to the product of the ML-degrees for $I({B_{T'}})$ and $ I({B_{S_m}})$. 
\end{cor}

\begin{proof}
The matrix $\mathcal{A}$ is the $3\times 3$ identity matrix, and hence has full rank. Thus, from \cite[Theorem~5.5]{AKK19}, the ML-degree of the toric fiber product of two toric models is the product of the ML-degrees of the models.  Thus, \Cref{thm:toricfibertree} implies that the ML-degree of  $I(B_T)$ is equal to the product of the ML-degrees of $I(B_{T'}^{\star})$ and $ I(B_{S_m}^{\star})$. This is equal to the product of the ML-degrees of $I({B_{T'}})$ and $ I({B_{S_m}})$ by Remark \ref{rem:MLdegreTFP}.  
\end{proof}

\section{Reciprocal ML-degree of star tree models}\label{s:star_tree}

A \emph{star tree} $S_n$ is a tree on leaves $\{0,\ldots, n\}$ with a unique internal vertex.  We compute the reciprocal ML-degree of star tree models in the following theorem. This serves as the basis of induction in the proof of the main theorem.

\begin{thm}
\label{thm:dualMLstarT}
The reciprocal maximum likelihood degree of the Brownian motion star tree model on $n+1$ leaves is equal to $2^n-n-1$.
\end{thm}

In preparation of the proof,
let $I_n$ be the defining ideal of the toric variety $\Lcal_{S_n}^{-1}$ in the $\bfp$ coordinates as given in \Cref{eq:changeOfCoordinates}.  By Proposition \ref{prop:RowspansEqual}, the ideal $I_n$ is equal to the ideal $I(B_{S_n})$, where the matrix $B_{S_n} \in \ZZ^{(n+1) \times \binom{n+1}{2}}$ as defined in \Cref{s:toricfiber} has columns $\{\mathbf{e}_i + \mathbf{e}_j \in \ZZ^{n+1}\mid 0 \leq i < j \leq n\}$.  In other words, the ideal $I_n$ is the toric ideal of the second hypersimplex, for which the following facts are well-known.

\begin{thm}\label{thm:hypersimplex}
The following hold for the toric ideal $I_n$.
\begin{enumerate}[label=(\alph*)]
\item \label{thm:generators}\cite[Theorem 2.1]{DST1995} The ideal $I_n \subset \CC[\bfp]$ is generated by the quadrics 
\[
\label{eq:GrobnerBasis}
    p_{ij}p_{kl}-p_{ik}p_{jl}, \text{ for distinct  }  i,j,k,l \in \{0,1\dots,n\}.
\]
\item \label{thm:degreeOfIdeal} \cite[Theorem 2.3]{DST1995} The degree of $V(I_n)$, as a projective variety in $\mathbb P^{\binom{n+1}{2}-1}$, is equal to $2^{n}-n-1$.
\end{enumerate}
\end{thm}

Along with the above \Cref{thm:hypersimplex}, the following will be a key step in the proof of \Cref{thm:dualMLstarT}.

\begin{lemma}\label{lem:zerointersection}
The varieties $\cL_{S_n}^{\perp}$ and $\cLI_{S_n}$ in $\mathbb{S}^n$ intersect only at the zero matrix.
\end{lemma}

\begin{proof}
Let $K\in \mathbb S^n$ be in the intersection $\cL_{S_n}^{\perp} \cap \cLI_{S_n}$, and write $(p_{ij})_{0\leq i < j \leq n}$ for the resulting coordinates after the change of coordinates in \Cref{eq:changeOfCoordinates}.  Let $P$ be an $n\times n$ symmetric matrix with diagonal entries $p_{01}, \ldots, p_{0n}$ and the off-diagonal entries $p_{ij}$ for $1\leq i < j \leq n$.  

The equations for $K\in \Lcal_{S_n}^\perp$ in terms of coordinates in $P$, as previously computed in \Cref{Eqn:BrownianLinearSystem}, are equivalent to
\[
   p_{01}+\dots +p_{0n}=0
   \quad \text{and}\quad \sum\limits_{\substack{i=0\\ i \neq j}}^np_{ij}=0, \text{ for }  j=1,\dots,n.
\]
In other words, the trace of $P$ and every row sum of $P$ are zero.

The condition $K\in \Lcal_{S_n}^{-1}$ is equivalent to $P \in V(I_n)$, again by \Cref{thm:SUZmain}.
The explicit set of generators for $I_n$ given in \Cref{thm:hypersimplex} impose the following condition on the entries of $P$:
For $1\leq i<j\leq n$, define $Q_{ij}$ to be the $2\times (n-1)$ matrix obtained by
\begin{enumerate}[label=(\roman*)]
    \item taking the  i-th and j-th row of $P$ to make a $2\times n$ matrix,
    \item  then converting the square submatrix $\begin{bmatrix} p_{0i} & p_{ij}\\
    p_{ij} & p_{0j}\end{bmatrix}$ to  $\begin{bmatrix} p_{0i} & p_{ij}\\
    p_{0j} & p_{ij}\end{bmatrix}$,
    \item and then erasing the column  $\begin{bmatrix} p_{ij} \\
    p_{ij}\end{bmatrix}$. 
\end{enumerate}
For all $1\leq i<j\leq n$, the $2\times 2$ minors of $Q_{ij}$ belong to the set of generators for $I_n$ in \Cref{thm:hypersimplex}.
Since the row sums of $P$ must be zero, we have that both row sums of $Q_{ij}$ are equal to $-p_{ij}$. Thus, that the rank of $Q_{ij}$ is at most 1 implies that if $p_{ij} \neq 0$, then $p_{il}=p_{jl}$ for all $l=1,\dots,n$.
As a result, if we consider the graph $G$ on vertices $\{1,\dots,n\}$ where $(i,j)$ is an edge in $G$ if and only if $p_{ij}\neq 0$, we have:
\begin{enumerate}
    \item Connected components of $G$ are complete graphs, and 
    \item for any $i\neq j$ belonging to a common connected component of $G$, all the $p_{ij}$ share a common value. 
\end{enumerate}
Thus, after relabeling, the matrix $P$ is a block diagonal matrix, each block having the form of a $(m+1)\times (m+1)$ matrix:
\[
\begin{bmatrix}
-ma & a &\dots & a\\
a & -ma &\dots & a\\
\vdots & \vdots &\ddots & \vdots\\
a & a &\dots & -ma
\end{bmatrix}.
\]
Suppose there are many blocks, say of sizes $m_1+1,\dots, m_{\ell}+1$.  Take $Q_{ij}$ with $i=m_a$ and $j=m_b$, for $1\leq a<b\leq \ell$. Then 
\[
Q_{ij}=\begin{bmatrix}
0 ~ \dots ~0 & a ~ \dots ~ a ~ m_a a& 0 ~ \dots ~ 0  & 0~ \dots ~ 0 \\
0 ~ \dots ~0 & 0 ~ \dots ~ 0 ~ m_b b& b ~ \dots ~ b  & 0~ \dots ~ 0 
\end{bmatrix}.
\]
For $Q_{ij}$ to have all vanishing $2 \times 2$ minors, at least one of $a$ and $b$ need be zero.  Hence, there can be at most one block with non-zero entries. If there is only one block, then $\trace(P)=0$ implies that $a=0$ and that $P$ is the zero matrix.  We thus conclude that $P$ is the zero matrix. 
\end{proof}

\begin{proof}[Proof of \cref{thm:dualMLstarT}]
For $T = S_n$, \Cref{Thm:DegreesEqual} states that the reciprocal ML-degree of $\Mcal_{S_n}$ is equal to to degree of $V(I_n) \cap V(\langle A_T \bfp - A_T \bfu \rangle)$ as an affine subscheme of $\CC^{\binom{n+1}{2}}$ for a generic $\bfu$.  Let us consider the intersection of their respective projective closures.  That is, we homogenize the ideals $I_n$ and  $\langle A_T \bfp - A_T \bfu \rangle\subset \CC[p_{ij} \mid 0\leq i<j\leq n]$ by an extra variable $p_\star$.  As the ideal $I_n$ is already homogeneous, the resulting homogenization $\overline{I_n}$ is the extension of $I_n$ in $\CC[p_\star,p_{ij}\mid 0\leq i < j \leq n]$, and $\langle A_T \bfp - A_T \bfu \rangle$ homogenizes to $\langle A_T \bfp - p_\star A_T \bfu \rangle$.
As projective varieties in $\mathbb P^{\binom{n+1}{2}}$, the intersection of $V(\overline{I_n})$ with the linear subvariety $V(\langle A_T \bfp - p_\star A_T \bfu \rangle)$ is the degree of $V(\overline{I_n})$.  Since $V(\overline{I_n})$ is the projective cone over $V(I_n)$ considered as a projective variety in $\mathbb P^{\binom{n+1}{2}-1}$, we thus conclude from \Cref{thm:hypersimplex}.\ref{thm:degreeOfIdeal} that the degree of the intersection $V(\overline{I_n}) \cap V(\langle A_T \bfp - p_\star A_T \bfu \rangle)$ is $2^n-n-1$.

It remains only to show that the intersection $V(\overline{I_n}) \cap V(\langle A_T \bfp - p_\star A_T \bfu \rangle)$ has no point in the hyperplane at infinty $\{p_\star = 0\}$.  Recall from the proof of \Cref{Thm:DegreesEqual} that in the $\bfp$-coordinates, $\bfp \in \Lcal_T^{\perp}$ if and only if $A_T \bfp = 0$. Thus when $p_\star = 0$, the equations defining the intersection are exactly the ones defining intersection $\Lcal_{T}^{-1} \cap \Lcal_{T}^\perp$, which only consists of the zero matrix by Lemma~\ref{lem:zerointersection}.  Hence, the intersection $V(\overline{I_n}) \cap V(\langle A_T \bfp - p_\star A_T \bfu \rangle)$ is empty if $p_\star = 0$, as desired.
\end{proof}

We can now prove the main result of the paper.

\begin{proof}[Proof of \cref{t:dualMLdegree}]
We induct on the number of  internal vertices of $T$. When $T$ has one internal vertex $v$, it is a star tree. So by Theorem~\ref{thm:dualMLstarT}, the dual ML-degree of $\Mcal_T$ is $2^{\operatorname{outdeg}(v)} - \operatorname{outdeg}(v) -1$.

Take a rooted tree $T$ with at least two internal vertices. Choose $\ell$ to be one of the internal vertices of $T$ that has only leaves as direct descendants. Let $h$ be the unique direct ancestor of $\ell$. Take $S_{\operatorname{outdeg}(\ell)}$ to be the rooted star tree with internal vertex $\ell$, root leaf $h$, and the remaining leaves  are exactly the descendants of $\ell$ in $T$. Take $T'$ to be the rooted tree obtained by removing from $T$ all leaves descendent of $\ell$.   Identifying   $h$ and $\ell$ in $S_{\operatorname{outdeg}(h)}$ and $T'$ gives back the tree $T$. Moreover, we have that $\In(T) = \In(T') \cup \{ \ell \}$. 
By \Cref{cor:MLdegreToricFiberProduct} and the inductive hypothesis, the dual ML-degree of $\Mcal_T$ is
\begin{align*}
\rmld(\Mcal_T)&=(2^{\operatorname{outdeg}(\ell)} - \operatorname{outdeg}(\ell) -1) \prod_{v \in \In(T')}(2^{\operatorname{outdeg}(v)} - \operatorname{outdeg}(v) - 1) \\&= \prod_{v \in \In(T)}(2^{\operatorname{outdeg}(v)} - \operatorname{outdeg}(v) - 1),
\end{align*}
as desired.
\end{proof}

\bibliography{paper}

\end{document}